\documentclass[12pt]{amsart}
\usepackage{tikz}
\usepackage{amsmath}
\usepackage{amsxtra}
\usepackage{amscd}
\usepackage{amsthm}
\usepackage{amsfonts}
\usepackage{amssymb}
\usepackage {pstricks}
\usepackage{pstricks,pst-node}
\usepackage[all]{xy}

\newtheorem{theorem}{Theorem}[section]

\newtheorem{lemma}[theorem]{Lemma}

\newtheorem{corollary}[theorem]{Corollary}

\theoremstyle{definition}
\newtheorem{definition}[theorem]{Definition}
\newtheorem{example}[theorem]{Example}

\theoremstyle{remark}
\newtheorem{remark}[theorem]{Remark}

\numberwithin{equation}{section}

\newcommand{\mc}{\mathcal}

\newcommand{\C}{{\mathbb C}}

\newcommand{\Z}{{\mathbb Z}}
\newcommand{\N}{{\mathbb N}}

\newcommand{\bU}{{\mathbb U}}
\newcommand{\bA}{{\mathbb A}}

\newcommand{\CB}{{\mathcal B}}

\newcommand{\CF}{{\mathcal F}}
\newcommand{\CP}{{\mathcal P}}

\newcommand{\CT}{{\mathcal T}}

\newcommand{\id}{{\rm{id}}}

\newcommand{\mf}{\mathfrak}
\newcommand{\fg}{{\mf g}}

\newcommand{\cB}{\mc B}

\newcommand{\be}{\begin{equation}}
\newcommand{\ee}{\end{equation}}

\newcommand{\inv}{^{-1}}

\newcommand{\ul}{\underline}

\newcommand{\lr}{\longrightarrow}
\newcommand{\wt}{\widetilde}

\newcommand{\End}{{\rm{End}}}
\newcommand{\Hom}{{\rm{Hom}}}

\newcommand{\GL}{{\rm{GL}}}
\newcommand{\Sym}{{\rm{Sym}}}

\advance\headheight by 2pt

\newcommand{\fso}{{\mathfrak {so}}}

\newcommand{\osp}{{\mathfrak {osp}}}

\newcommand{\Sp}{{\rm Sp}}
\newcommand{\Or}{{\rm O}}
\newcommand{\ot}{\otimes}

\newcommand{\OSp}{{\rm OSp}}

\newcommand{\SO}{{\rm SO}}
\newcommand{\ve}{{\varepsilon}}

\begin{document}

\normalfont

\title[Invariants and an enhanced Brauer category]{Invariants of the special orthogonal group and an enhanced Brauer category}

\author{G.I. Lehrer and R.B. Zhang}
\thanks{This research was supported by the Australian Research Council}
\address{School of Mathematics and Statistics,
University of Sydney, N.S.W. 2006, Australia}
\email{gustav.lehrer@sydney.edu.au, ruibin.zhang@sydney.edu.au}
\begin{abstract}
We first give a short intrinsic, diagrammatic proof of the First Fundamental Theorem of invariant theory (FFT)
 for the special orthogonal group
$\SO_m(\C)$, given the FFT for $\text{O}_m(\C)$. We then define, by means of a 
presentation with generators and relations, an enhanced Brauer category $\wt\cB(m)$ by adding a single generator
to the usual Brauer category $\CB(m)$, together with four relations. 
We prove that our category $\wt\cB(m)$ is actually (and remarkably) {\em equivalent} to the category
of representations of $\SO_m$ generated by the natural representation. The FFT for $\SO_m$ amounts to the surjectivity 
of a certain functor $\CF$ on $\Hom$ spaces, while the Second Fundamental Theorem for $\SO_m$ says simply
 that $\CF$ is injective on $\Hom$ spaces. This theorem provides a diagrammatic means
of computing the dimensions of spaces of homomorphisms between tensor modules for $\SO_m$ (for any $m$). These methods will be applied to the
case of the orthosymplectic Lie algebras $\osp(m|2n)$, where the super-Pfaffian enters, in a future work.
\end{abstract}
\date{\today}
\maketitle
\section{Introduction}
\subsection{History} If a group $H$ acts on a vector space $W$, the element $w\in W$ is said to be {\em invariant} if $hw=w$
for all $h\in H$. The set of invariants $W^H$ is a subspace of $W$, and in general the problem of finding a spanning
set for $W^H$ is known as the first fundamental problem of invariant theory, and that of describing all
relations among the elements of the spanning set is the second fundamental problem.
When $W$ has extra structure, e.g. if it is an associative algebra, or a commutative algebra, the first and second problems
are recast as asking for generators and relations respectively in the appropriate category. In the two examples given, one
would ask for algebra generators of the invariants, together with generators of the ideal of relations among them.

If $V=\C^m$, $H=\GL(V)$ and $W=\End_\C(V^{\ot r})$ for some positive integer $r\geq 1$, Schur \cite{Sch1} completely solved the 
fundamental problems, and proved what are now known as the first and second fundamental theorems (FFT and SFT) of invariant theory.
Note that in this case, $W$ is an associative, non-commutative (if $r>1$) algebra. 
It was subsequently shown that his theorem is equivalent to solving
the corresponding problem for the commutative algebra $\C[W]$, where $W=V^{\oplus r}\oplus {V^*}^{\oplus s}$,
which is just the symmetric algebra on $W^*$. These results are
 well known, and we shall not rehearse them here. Suffice it to say that they have led to applications in the representation theory of Lie groups
 (cf. \cite{W}) and to an enormous literature in what has become known as Schur-Weyl duality, in many and varied settings.
 
 When $V$ above is endowed with a non-degenerate symmetric (resp. skew) bilinear form $(\;,\;)$ and $H=\Or(V)$ 
(resp. $\Sp(V)$) is the full symmetry group of $(V,(\;,\;))$, 
Weyl \cite{W} has proved results which could be interpreted as versions of the FFT and SFT for 
the $H$ action on the commutative algebra $\C[V^{\oplus r}]$. His FFT asserts that the space of invariant functions
is generated as a commutative algebra by contractions using the given form, while his SFT asserted that all relations among these generators 
are generated by the obvious determinantal condition.

 This work has an important reinterpretation in terms of diagrams, 
which was given in 1937 by Brauer \cite{B}. With $V$ and $H$ as in the previous paragraph,
he defined certain algebras $\CB_r$ in terms of diagrams, as well as 
homomorphisms $\eta_r:\CB_r\lr \End_{H}(V^{\ot r})$ ($r=1,2,3,\dots$) whose surjectivity 
 is equivalent to the FFT in this setting. Brauer left open the SFT in this formulation,
which amounts to describing the kernel of $\eta_r$, an ideal of $\CB_r$. This
question has only recently been resolved \cite{LZ4, LZ5} by the present authors, who produced
 an explicit ideal generator in $\CB_r$ for the kernel of $\eta_r$, and in this way, provided a presentation 
of $\End_{H}(V^{\ot r})$.
 
 The above circle of ideas are most efficiently expressed in the language of the Brauer category \cite{LZ5}.
 
 There are now many vast generalisations of this work in which the base ring $\C$ is replaced by other domains, 
 and in which `group' may be replaced by Lie algebra, quantum group, or another algebraic object which has
an action or co-action (see, e.g., below).
 These generalisations are often of central importance in addressing difficult questions concerning multiplicities and
 character formulae and are related to geometric questions such as intersection cohomology of the flag varieties.
 They are part of the vast subject of `invariant theory', which includes geometric invariant theory, a subject distilled
 by Mumford \cite{Mu} from many persistent mathematical themes through many ages. If $H$ is an algebraic
 group over $\C$ acting on an affine variety $X$, then $H$ acts on the coordinate ring $\C[X]$, and the properties of the 
 `orbit variety' $X// H$ are reflected in the properties of the ring $\C[X]^H$ of $H$-invariants.   
 
 In this work we shall be concerned with a particular issue in the orthogonal case, which arises in general
in the study of the invariants of the orthosymplectic Lie algebras. Consider the situation described above, when the form $(\;,\;)$
is symmetric. Taking $W=\End_\C(V^{\ot r})$ as above, but with $H=\SO(V)$ rather than $\Or(V)$, evidently there are more invariants, but 
the structure of the ring of invariants has always been somewhat imperfectly understood. 
 It is this case and its generalisations upon which this work is focussed.

\subsection{This work}  In recent work \cite{DLZ, LZ6} we have proved a first fundamental theorem (FFT) of invariant theory for 
the orthosymplectic group super scheme $\OSp(V)$, where  $V$ is a super space over $\C$. The proof is geometric and
applies to the classical groups as special cases; it
does not involve the Capelli identities which appear in traditional proofs (cf. \cite[Thm 10.2A]{W}, \cite[App. F]{FH}
or \cite{Ri}).

These results all have a categorical formulation in terms of the Brauer category. 

In this note we give a short diagrammatic argument to deduce the FFT for $\SO(V)$ from that for $\Or(V)$. This is direct, and involves no
induction on dimension, which appears in traditional proofs such as those in \cite{FH, Ri}. We then
define an enhancement of the Brauer category, which is still described in terms of a presentation using diagrams,
with the generators and relations which define the Brauer category (cf. \cite{LZ5}) plus one extra generator and several extra relations 
which involve it. We prove a FFT for $\SO_m$ for our new category, but the surprising fact to emerge is that 
one can prove that the new relations
include the generator of the kernel of the functor from the Brauer category to the category of tensor representations.
This implies that the maps on $\Hom$ spaces are all automatically isomorphisms, i.e. that our new category is equivalent
to the category of tensor representations of $\SO_m$. In particular, the dimensions of the endomorphism 
algebras may be computed formally in our category.

The calculation which yields that the ideal of relations we give here includes the kernel of the maps from the $\Hom$
spaces in our abstract category to those in the category of tensor representations of $\SO_m$ appears for $m=3$
in \cite{BE}, which was part of our inspiration for this work.

Although our results here are framed for spaces over $\C$, they actually apply without extra effort 
to any field of characteristic $p>m=\dim V$, and with a little more effort to any field of odd characteristic.

In a future work we intend to apply similar methods to the case of the orthosymplectic Lie super algebra
$\osp(m|2n)$. In this case, the space of super Pfaffians (cf. \cite{LZ8}) will come into play as the harmonic generators of
those invariants of $\osp(m|2n)$ which are semi invariant for the super group scheme $\OSp(m|2n)$.

\section{Preliminaries}
Let $V=\C^m$ be equipped with a non-degenerate symmetric bilinear form $(\;,\;)$. Let 
$G:=\Or(V)$ be its isometry group, and $G'=\SO(V)$ the identity component of $G$. We refer to
$V$ as the natural module for $G$ or $G'$, and we shall be concerned with the invariants $(V^{\ot r})^{G'}$. 

Assuming the  first fundamental theorem (FFT) of invariant theory for $G$,
we give a short proof of the FFT of the invariant theory for the 
special orthogonal group $G'$, or equivalently, of the orthogonal Lie algebra
$\fg:=\fso(V)$,
in the context of tensors powers of the natural module.  That is,
we shall give a spanning set for $(V^{\ot r})^{G'}$.
The result is of course classical, but our proof is particularly short and 
depends only on diagrammatic ideas.
It also extends to some extent to the super case, i.e. to the case where $V$ is a super space and $G=\OSp(V)$
\cite{LZ6}. 

If $M$ is any finite dimensional $G$-module, then {
$
M^\fg=M^{G'}=M^G\oplus M^{G, \det},
$
where 
\[
M^{G, \det}=\{m\in M\mid g(m)=\det(g) m, \  \forall  g\in G \}.
\]
}
Our aim is to explicitly describe $M^{G'}$ when $M=V^{\ot r}$.
We assume that $(V^{\otimes r})^G$ is known, and given by the FFT for $G$
as stated in Theorem \ref{thm:fft} below; hence we need only consider 
$(V^{\otimes r})^{G, \det}$. To state our theorem, we need the element $\check C\in V\ot V$ and its dual $\hat C$ 
introduced in \S\ref{ss:brauer} below. An element $D\in V^{\ot r}$  is {\em harmonic} if 
\[
(\hat C\ot(\id_V)^{\ot r-2})\circ\sigma\circ D=0\text{ for all }\sigma\in\Sym_r.
\]

\begin{theorem} \label{thm:so-inv}  Assume that for some integer $r_c>0$,
there exists a harmonic element $\Lambda$ in $(V^{\otimes r_c})^{G, \det}$ satisfying the condition $(\Lambda, \Lambda)\ne 0$. 
Then  $(V^{\otimes r})^{G, \det}\ne 0$ if and only if $r-r_c\ge 0$ is even, and in this case, 
\[
(V^{\otimes r})^{G, \det}= \C\Sym_r\left(\Lambda\otimes \check{C}^{\otimes \frac{r-r_c}{2}}\right).
\]

\end{theorem}
\begin{remark}
Since any nonzero element  $\Lambda\in\wedge^{\dim V}V$ evidently
satisfies the conditions of Theorem \ref{thm:so-inv}, we recover
the FFT of invariant theory of the special orthogonal group (see \cite[Appendix F]{FH} and \cite[\S 10.2]{KP}).
\end{remark}

We prepare for the proof by explaining a diagrammatic way of viewing the problem.
\section{The category $\CT$.}

{
Let $\CT$ be the  full subcategory of the category of finite dimensional modules for  $G'=\SO(V)$ with objects $V^{\otimes r}$, where $r\in\N$. For any $i,j\in\N$,
let 
 \[
\CT_i^j:=\Hom_{G'}(V^{\ot i},V^{\ot j}).
\]}
Note that $\CT$ is naturally a tensor category: if $A\in\CT_i^j$ and $B\in \CT_r^s$, then $A\ot B\in \CT_{i+r}^{j+s}$
is defined in the obvious way. We also have a natural duality on $\CT$ arising from the fact that the form $(\;,\;)$ extends
naturally to a non-degenerate symmetric bilinear form, also denoted $(\;,\;)$ on $V^{\ot r}$ for any $r$, so that 
the dual space $(V^{\ot r})^*$ is naturally identified with  $V^{\ot r}$. If $\ul v\in V^{\ot r}$,
denote by $\ul v^*$ the element of $(V^{\ot r})^*$ defined by $\ul v^*(\ul w)=(\ul v,\ul w)$  for all $\ul w\in V^{\ot r}$.
Note that if $g\in\Or(V)$, then $g(\ul v^*)=(g\ul v)^*$.

For $A\in\CT_i^j$, the dual $A^*\in\CT_j^i$ is defined by 
\[
(A\ul v,\ul w)=(\ul v, A^*\ul w)
\]
for all $\ul v\in V^{\ot i}$ and $\ul w\in V^{\ot j}$.

\subsection{Diagrammatics} If $A\in\CT_i^j$, we denote $A$ by a diagram $D$:

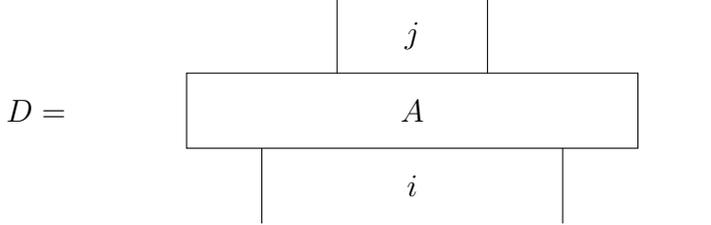
\begin{tikzpicture}
\node at (0,1) {$D=$};
\node at (5,1) {$A$};
\node at (5,2) {$j$};
\node at (5,0) {$i$};
\node at (9,1) {$.$};
\draw (2,0.5) rectangle (8,1.5);
\draw (4,1.5)--(4,2.5);\draw (6,1.5)--(6,2.5);\draw (3,0.5)--(3,-0.5);\draw (7,0.5)--(7,-0.5);

\end{tikzpicture}

Then $A^*$ is represented by the reflection $D^*$ in a horizontal of the above diagram:

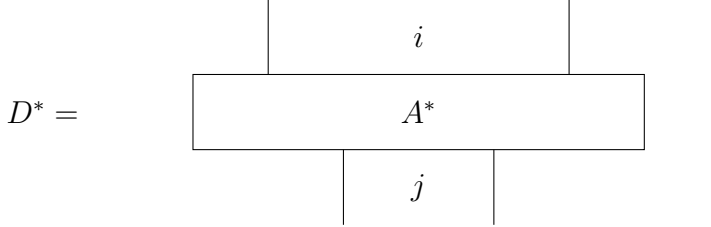
\begin{tikzpicture}
\node at (0,1) {$D^*=$};
\node at (5,1) {$A^*$};
\node at (5,2) {$i$};
\node at (5,0) {$j$};
\node at (9,1) {$.$};
\draw (2,0.5) rectangle (8,1.5);
\draw (3,1.5)--(3,2.5);\draw (7,1.5)--(7,2.5);\draw (4,0.5)--(4,-0.5);\draw (6,0.5)--(6,-0.5);

\end{tikzpicture}

Composition of morphisms is thought of as concatenation of diagrams.

An element $\ul v\in \left(V^{\ot r}\right)^{G'}$ may be thought of as the map in $\CT_0^r$ which sends $1\in \C$ to $\ul v$. Then $\ul v^*$ is the map
in $\CT_r^0$ which takes $\ul w$ to $(\ul v,\ul w)$. Diagramatically:

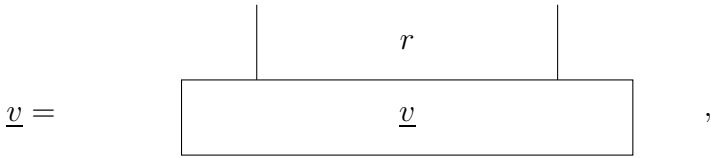
\medskip

\begin{tikzpicture}
\node at (0,1) {$\ul v=$};
\node at (5,1) {$\ul v$};
\node at (5,2) {$r$};
\node at (9,1) {$,$};
\draw (2,0.5) rectangle (8,1.5);
\draw (3,1.5)--(3,2.5);\draw (7,1.5)--(7,2.5);

\end{tikzpicture}

\medskip

and 

\medskip

\begin{tikzpicture}
\node at (0,1) {$\ul v^*=$};
\node at (5,1) {$\ul v^*$};
\node at (5,0) {$r$};
\node at (9,1) {$.$};
\draw (2,0.5) rectangle (8,1.5);
\draw (3,0.5)--(3,-0.5);\draw (7,0.5)--(7,-0.5);

\end{tikzpicture}

{
If $\ul v,\ul w\in\left(V^{\ot r}\right)^{G'}$, 
}
we may form the compositions $\ul v\circ\ul w^*\in\CT_r^r$ and $\ul w^*\circ\ul v\in\CT_0^0$,
and it is easily verified that 
\be\label{eq:comp}
\begin{aligned}
\ul v\circ\ul w^*:\ul x&\mapsto (\ul w,\ul x)\ul v\text{ and }\\
\ul w^*\circ\ul v:1&\mapsto (\ul w,\ul v).\\
\end{aligned}
\ee

In terms of diagrams, \eqref{eq:comp} may be written

\medskip
{
\begin{tikzpicture}
\node at (5,3) {$\ul v$};
\node at (5,4) {$r$};

\draw (2,2.5) rectangle (8,3.5);


\draw (3,3.5)--(3,4.5);\draw (7,3.5)--(7,4.5);


\node at (10,2) {$=\phi_{\ul w^*,\ul v}\in\CT_r^r,$};

\node at (5,1) {$\ul w^*$};
\node at (5,0) {$r$};

\draw (2,0.5) rectangle (8,1.5);



\draw (3,0.5)--(3,-0.5);\draw (7,0.5)--(7,-0.5);


\end{tikzpicture}
}

\medskip

\noindent {where  $\phi_{\ul w^*,\ul v}(\ul x)=(\ul w,\ul x)\ul v$ for all $\ul x\in V^{\ot r}$,  and}

\bigskip

\begin{tikzpicture}
\node at (5,3) {$\ul w^*$};
\draw (2,2.5) rectangle (8,3.5);
\draw (3,1.5)--(3,2.5);\draw (7,1.5)--(7,2.5);
\node at (10,2) {$=(\ul v,\ul w)\in\CT_0^0=\C.$};

\node at (5,1) {$\ul v$};
\draw (2,0.5) rectangle (8,1.5);

\end{tikzpicture}

\subsection{The Brauer subcategory $\CB\subset\CT$} \label{ss:brauer}
Let $e_1,\dots,e_m$ be an orthonormal basis of $V$. The element
$\check C:=\sum_i e_i\ot e_i\in V\ot V$ is independent of the basis, and $\Or(V)$-invariant.
Regarded as an element of $\CT_0^2:=\Hom_{\CT}(0,2)$, its dual is $\check C^*=\hat C\in\CT_2^0$, where 
for $v,w\in V$, $\hat C(v\ot w)=(v,w)$ (since $(\check C,v\ot w)=(v,w)$). Note that the relations
\eqref{eq:comp} show that $\hat C\circ\check C=m$ and $\check C\circ\hat C:v\ot w\mapsto (v,w)\check C$.

We recall the definition of the Brauer category $\CB=\CB(m)$ (over $\C$) from \cite{LZ5}. It has objects $\N$ and morphisms
which are generated by the four morphisms $I,U,A=U^*$ and $X$ in $\CB_1^1, \CB_0^2,\CB_2^0$ and $\CB_2^2$ respectively.
Moreover we have a functor $\CF$ from $\CB$
to $\CT$ which takes $I$ to $\id_V$, $U$ to $\check C$,  $A$ to $\hat C$
and $X$ to $\tau\in\CT_2^2$ given by $\tau(v\ot w)=w\ot v$.  We abuse language by referring to the image of this functor as the Brauer subcategory 
$B:=\CF(\CB)\subset\CT$, and to its morphisms $B_i^j$ as Brauer morphisms. Note that $\C\Sym_r\subset B_r^r$.
Diagrammatically, the morphisms $\id_V,\check C,\hat C$ and $\tau$ are respectively the images of the following diagrams:

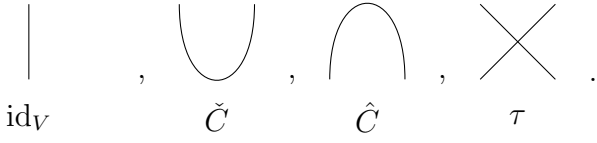
\begin{tikzpicture}
\draw (0,0)--(0,1);
\draw (2,1).. controls (2,-0.35) and (3,-0.35).. (3,1);\draw (4,0).. controls (4,1.35) and (5,1.35).. (5,0);
\node at (1.5,0){,};\node at (3.5,0){,};\node at (5.5,0){,};
\draw (6,0)--(7,1);\draw (7,0)--(6,1);
\node at (0,-.5){$\id_V$};\node at (2.5,-.5){$\check C$};\node at (4.5,-.5){$\hat C$};\node at (6.5,-.5){$\tau$};
\node at (7.5,0){.};
\end{tikzpicture}

The FFT for $\Or(V)$ may be stated as follows.
\begin{theorem}\label{thm:fft} In the above notation, the functor $\CF:\CB(m)\lr \CT^{\Or(V)}$ is full,
where $\CT^{\Or(V)}$ is {
the subcategory of $\CT$ with $\Or(V)$-invariant morphisms}. 
In particular, for any non-negative integers $i,j$, we have
\be\label{eq:fft1}
(\CT_i^j)^{\Or(V)}=B_i^j.
\ee
Equivalently, $(V^{\ot r})^{\Or(V)}=0$ if $r$ is odd, while if $r=2d$ is even, then 
\be\label{eq:fft2}
(V^{\ot 2d})^{\Or(V)}=\C\Sym_{2d}(\check C^{\ot d}).
\ee
\end{theorem}

\begin{remark}\label{rem:fftdiag}
(i) In terms of diagrams, the theorem states that $(V^{\ot r})^{\Or(V)}\subseteq\CT_0^r$ is spanned by diagrams of the 
following form.

\bigskip
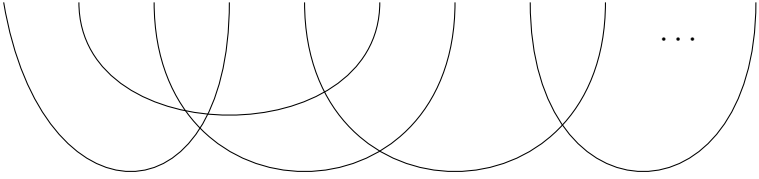
\begin{tikzpicture}
\draw (0,3).. controls (.5,0) and (3,0).. (3,3);
\draw (1,3).. controls (1,1) and (5,1).. (5,3);
\draw (2,3).. controls (2,0) and (6,0).. (6,3);
\draw (4,3).. controls (4,0) and (8,0).. (8,3);
\draw (7,3).. controls (7,0) and (10,0).. (10,3);
\node at (9,2.5) {$\cdots$};
\end{tikzpicture}

(ii) It follows from Theorem \ref{thm:fft} that if $\ul\xi\in  (V^{\ot 2d})^G$,
where $G=\Or(V)$ (so that $\ul\xi^*\in\CT_{2d}^0$), and if $\ul\xi^*\circ D= 0$
for each diagram of the above form, then $\ul\xi=0$. 

(iii) The diagram shown represents an element of $\CT_0^r\simeq(V^{\ot r})^G$. Its dual in $\CT_r^0$, 
which is the diagram obtained by reflecting the one shown in a horizontal line, applied to $\ul v=v_1\ot\dots\ot v_r$
is the product of the contractions $(v_i,v_j)$ over pairs $(i,j)$ lying on a common arc.
\end{remark}

\section{Proof of Theorem \ref{thm:so-inv}.}

Suppose we have $\Lambda\in (V^{\ot r_c})^{G,\det}$ such that $(\Lambda,\Lambda)\neq 0$ and $\Lambda$ is harmonic.
Let  $\pi_{\Lambda}: V^{\otimes r_c} \longrightarrow\C\Lambda\subseteq V^{\ot r_c}$ be the
$G$-module homomorphism given by $\pi_\Lambda(\ul w)=\frac{(\Lambda, \ul w)}{(\Lambda, \Lambda)}\Lambda$ for any 
$\ul w\in V^{\otimes r_c}$.  For any integer $r\ge 0$,
let $\pi_{\Lambda, r} =\pi_{\Lambda}\otimes \id_{V^{\otimes r}}$; then 
$\pi_{\Lambda,r}^2=\pi_{\Lambda,r}$, so that $\pi_{\Lambda,r}$ is the orthogonal projection from $V^{\ot (r_c+r)}$ to 
$\C\Lambda\ot V^{\ot r}=\Lambda\ot V^{\ot r}$.  
\begin{proof}[Proof of Theorem \ref{thm:so-inv}]

Since taking $G$-fixed points of  $G$-modules
is an exact functor, we have an isomorphism 
\be\label{eq:surject}
 (\Lambda\otimes  V^{\otimes r} )^G\simeq \pi_{\Lambda, r}\left((V^{\otimes(r_c+ r)} )^G\right).
\ee


The FFT for $G$ enables us to determine the right hand side of \eqref{eq:surject}
by using the harmonicity of $\Lambda$. 

We show first diagrammatically that $(\Lambda\otimes  V^{\otimes r} )^G=0$ if $r<r_c$. 
For this, observe  that $(\Lambda,\Lambda)\pi_{\Lambda,r}\in\CT_{r_c+r}^{r_c+r}$ is represented by the
diagram

\smallskip

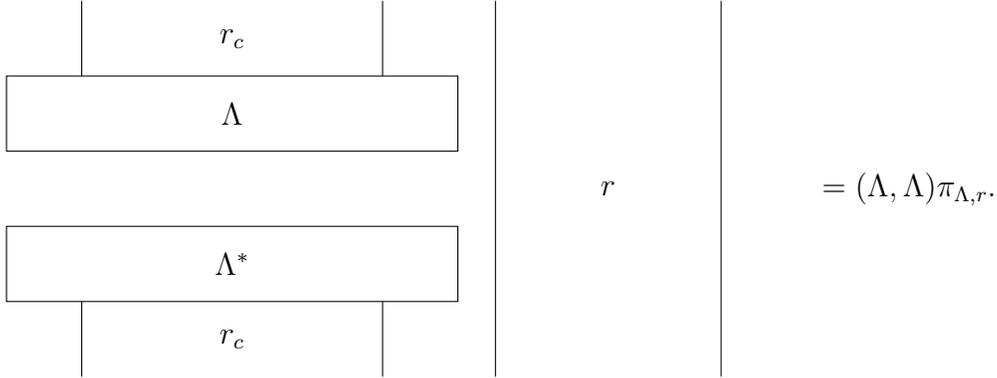
\begin{tikzpicture}
\node at (5,3) {$\Lambda$};
\node at (5,4) {$r_c$};
\draw (2,2.5) rectangle (8,3.5);
\draw (3,3.5)--(3,4.5);\draw (7,3.5)--(7,4.5);
\node at (14,2) {$=(\Lambda,\Lambda)\pi_{\Lambda,r}.$};
\node at (10,2) {$r$};
\node at (5,1) {$\Lambda^*$};
\node at (5,0) {$r_c$};
\draw (2,0.5) rectangle (8,1.5);
\draw (3,0.5)--(3,-0.5);\draw (7,0.5)--(7,-0.5);
\draw (8.5,-0.5)--(8.5,4.5);\draw (11.5,-0.5)--(11.5,4.5);

\end{tikzpicture}

\smallskip

Now if $r<r_c$,
it is evident that in any composition of the form $(\Lambda,\Lambda)\pi_{\Lambda,r}\circ D$, where
$D$ is the diagram in Remark \ref{rem:fftdiag},
 there will be an arc of the lower diagram which has both ends joined to $\Lambda^*$. 
By duality, the fact that $\Lambda$ is harmonic implies that for any $\sigma\in\Sym_{r_c}$, we have $\Lambda^*\circ\sigma\circ \check C=0$.
It follows that by harmonicity of $\Lambda$, if $r<r_c$, any such composition is zero, and by Remark \ref{rem:fftdiag} (ii),
it follows that any element of $(\Lambda\ot V^{\ot r})^G$ is zero.

Moreover, it follows by a straightforward manipulation of diagrams, that if $r\geq r_c$, then again applying 
the FFT for $G$ in the form \eqref{eq:fft2}, the right side of \eqref{eq:surject}
is a linear combination of the diagrams $\Lambda\ot [\sigma\circ(\Lambda\ot \check C^{\ot\frac{r-r_c}{2}})]$,
where $\sigma\in\Sym_r$. That is, if $r\geq r_c$, then as elements of $(V^{\ot (r+r_c)})^G\subset\CT_0^{r_c+r}$,
{
\be\label{eq:a}
\pi_{\Lambda,r}\left((V^{\ot (r+r_c)})^G\right)=\Lambda\ot[\C\Sym_r\circ\left(\Lambda\ot (V^{\ot(r-r_c)})^G\right)].
\ee
}
Next observe that evidently
\be\label{eq:b}
\Lambda\otimes (V^{\otimes r})^{G, \det}= (\Lambda\otimes V^{\otimes r})^G=\pi_{\Lambda,r}\left((V^{\ot (r+r_c)})^G\right). 
\ee
 It follows from \eqref{eq:a} and \eqref{eq:b}, as well as the above analysis of the case $r<r_c$, that 
\be\label{eq:c}
(V^{\otimes r})^{G, \det}=
\begin{cases}
0\text{ if $r<r_c$, and}\\ 
 \C\Sym_r\left(\Lambda\otimes (V^{\otimes (r-r_c)})^G\right), \quad \forall r\ge r_c .\\
 \end{cases}
\ee
Applying the FFT  for $G$ \eqref{eq:fft2} to identify $(V^{\ot r-r_c})^G$, we obtain the theorem. 
\end{proof}

\begin{example}\label{ex:1}
Let $e_1,\dots,e_m$ be an orthonormal basis of $V$, and take $\Lambda=e_1\wedge\dots\wedge e_m=
A_m(e_1\ot\dots\ot e_m)$ where $A_m=(m!)\inv\sum_{\sigma\in\Sym_r}\ve(\sigma)\sigma\in B_m^m$. Then  
$\Lambda$ is evidently harmonic, and $(\Lambda,\Lambda)=(m!)\inv\neq 0$.
\end{example}

\begin{corollary}\label{cor:fftso}
Any element $\alpha\in(\CT_i^j)^{\SO(V)}$ is uniquely 
expressible in the form $\alpha=\alpha_1+\alpha_2$, where
$\alpha_1\in B_i^j$ and $\alpha_2$ is obtained from an element of $(V^{\ot r})^{\Or(V),\det}\subset (\CT_0^r)^{\SO(V)}$
by tensoring and composition with elements of $B$, i.e., with Brauer morphisms.
\end{corollary}

Example \ref{ex:1} and Theorem \ref{thm:so-inv} give a complete description of the space $(V^{\ot r})^{\Or(V),\det}$,
so that Corollary \ref{cor:fftso} is the FFT for $\SO(V)$.

\begin{remark}
The arguments above may be adapted to prove a corresponding result \cite{LZ8} 
 (where the super Pfaffian enters) for $\osp_{1|2n}$.
\end{remark}

\section{The enhanced Brauer category $\wt\CB(\delta)$.}

Let $R$ be a commutative ring and $\delta\in R$. We fix a positive integer $m$.
With this data, we shall define a tensor category $\CB(\delta)$, 
which contains a quotient of the usual Brauer category $\CB(\delta)$ \cite{LZ5} as a subcategory. We shall see that the 
relations we impose imply a relationship between $\delta$ and $m$, so that for each $m$ there are only finitely
many values of $\delta$ which make our relations consistent. 
Both categories have objects $\N=\{0,1,2,\dots\}$, morphisms which may be described diagrammatically, and an involution $^*$ 
which fixes the objects and reflects morphisms in a horizontal.

\subsection{Definition of $\wt\cB(\delta)$} We have seen that $\cB(\delta)$ may be presented as the category with 
object set $\N$ and morphisms which are generated by the four morphisms $I,U,A$ and $X$ under composition,
tensor product and duality, subject to certain relations, which are described in \cite{LZ5}.
In the definition below, we shall make extensive use of the total anti-symmetriser $\Sigma_r\in \cB_r^r$,
where $\cB=\cB(\delta)$. This is defined by 
\be\label{eq:sigr}
\Sigma_r=\sum_{\pi\in\Sym_r}\ve(\pi)\pi,
\ee 
and is depicted diagramatically as
\medskip

\begin{tikzpicture}
\node at (0,1) {$\Sigma_r=$};
\node at (5,1) {$r$};
\node at (4,2) {$...$};
\node at (6,2) {$...$};
\node at (5,2) {$r$};
\node at (5,0) {$r$};
\node at (4,0) {$...$};
\node at (6,0) {$...$};
\draw (2,0.5) rectangle (8,1.5);
\draw (3,1.5)--(3,2.5);\draw (7,1.5)--(7,2.5);
\draw (3,0.5)--(3,-.5);\draw (7,0.5)--(7,-.5);

\end{tikzpicture}

\begin{definition}\label{def:wbrcat} Let $R$ be a ring, $\delta\in R$ and $m\geq 2$ a positive integer.
The category $\wt \cB(\delta)$ has object set $\N$ and morphisms which are generated by $I,U,A,X$ and one new generator
$\Delta_m\in\wt\cB_0^m$, subject to the following relations, which describe the interaction of the new generator with the
Brauer morphisms.
\begin{enumerate}
\item The relations \cite[Thm. 2.6(2)]{LZ5} for the generators $I,U,A$ and $X$.
\item (Harmonicity) For each positive integer $r$ with $0\leq r\leq m-2$, $(I^{\ot r}\ot A\ot I^{\ot m-r-2})\circ\Delta_m=0$.
\item For each positive integer $r$ with $0\leq r\leq m-2$, $(I^{\ot r}\ot X\ot I^{\ot m-r-2})\circ\Delta_m=-\Delta_m$.
\item $\Delta_m\circ \Delta_m^*=\Sigma_m$.
\end{enumerate}
\end{definition}
The new generator $\Delta_m$ will be depicted diagrammatically (as a morphism from $0$ to $m$) as follows.

\bigskip
\centerline{
\begin{tikzpicture}
\draw (-1,0).. controls (0,-1) .. (1,0);
\draw (-1,0)-- (1,0);
\draw (-.7,0)-- (-.7,1);
\draw (.7,0)-- (.7,1);
\node at (0,.5) {$m$};
\node at (0,-.5) {$m$};
\node at (-.5,.5) {$...$};
\node at (.5,.5) {$...$};
\end{tikzpicture}}

The relations above have suggestive diagrammatical interpretations, which are helpful in performing computations in the category
$\wt\CB(\delta)$. For example, the relation (4) may be depicted diagrammatically as follows.

\bigskip
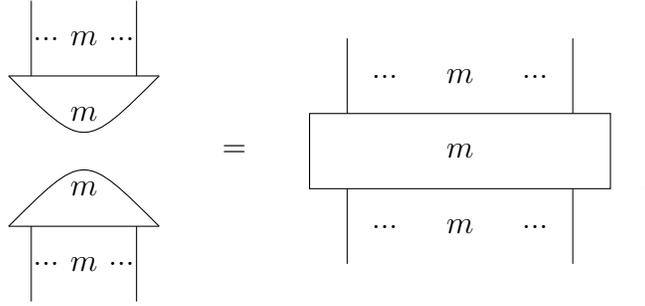
\begin{figure}
\begin{tikzpicture}
\draw (-1,0).. controls (0,-1) .. (1,0);
\draw (-1,0)-- (1,0);
\draw (-.7,0)-- (-.7,1);
\draw (.7,0)-- (.7,1);
\node at (0,.5) {$m$};
\node at (0,-.5) {$m$};
\node at (-.5,.5) {$...$};
\node at (.5,.5) {$...$};

\draw (-1,-2).. controls (0,-1) .. (1,-2);
\draw (-1,-2)-- (1,-2);
\draw (-.7,-2)-- (-.7,-3);
\draw (.7,-2)-- (.7,-3);
\node at (0,-2.5) {$m$};
\node at (0,-1.5) {$m$};
\node at (-.5,.-2.5) {$...$};
\node at (.5,-2.5) {$...$};

\node at (2,-1) {$=$};

\node at (5,-2) {$m$};
\node at (4,-2) {$...$};
\node at (6,-2) {$...$};
\node at (5,-1) {$m$};
\node at (5,0) {$m$};
\node at (4,0) {$...$};
\node at (6,0) {$...$};
\draw (3,-1.5) rectangle (7,-.5);
\draw (3.5,-1.5)--(3.5,-2.5);\draw (6.5,-1.5)--(6.5,-2.5);
\draw (3.5,0.5)--(3.5,-.5);\draw (6.5,0.5)--(6.5,-.5);
\node at (7.5,-1.5) {$.$};
\end{tikzpicture}
\caption{The relation 4.}
\end{figure}

\begin{remark}\label{rem:nonzero}
With our application to invariant theory in mind, we shall assume that the base ring $R$ is an integral domain, and that 
$m!\neq 0$ in $R$. 

It follows from condition (3) of Definition \ref{def:wbrcat} that $\Sigma_m\Delta_m=m!\Delta_m$, and hence by the above assumptions, 
that, if $\Sigma_m=0$, then $\Delta_m=0$. If $\Sigma_m=0$, the category $\wt\CB(\delta)$ therefore 
is just a quotient category of $\CB(\delta)$.

To avoid this degeneracy, we shall therefore assume that $\Sigma_m\neq 0$.
\end{remark}

\subsection{Some computations in the category $\wt\CB(\delta)$} In this subsection, we shall perform the computation which 
we shall later use to show that the kernel of the (surjective) map from $\wt\CB(m)_i^j$ to $(\CT_i^j)^{\SO_m}$ is zero.
We shall argue diagramatically, and recall the following result from \cite{LZ5}.

\begin{lemma}\label{lem:red}\cite[Lemma 2.1(1) and (2)]{LZ5} For all $r\geq 1$ we have the following relations 
in $\CB(\delta)$, and hence {\it a fortiori} in $\wt\CB(\delta)$.
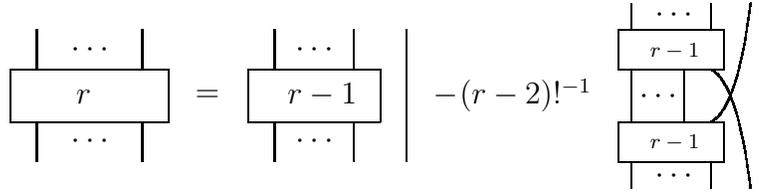
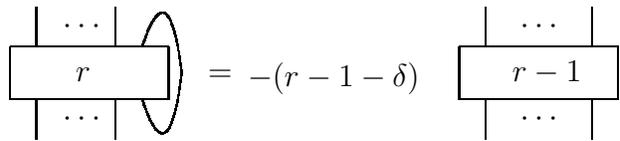
\begin{enumerate}
\item
\[
\begin{picture}(80, 60)(0, -30)
\put(0, 10){\line(1, 0){60}}
\put(0, -10){\line(1, 0){60}}
\put(0, 10){\line(0, -1){20}}
\put(60, 10){\line(0, -1){20}}
\put(25, -3){$r$}

\put(10, 10){\line(0, 1){15}}
\put(23, 15){$\cdots$}
\put(50, 10){\line(0, 1){15}}

\put(10, -10){\line(0, -1){15}}
\put(23, -20){$\cdots$}
\put(50, -10){\line(0, -1){15}}

\put(70, -2){$=$}
\end{picture}
\begin{picture}(80, 60)(-10, -30)
\put(0, 10){\line(1, 0){50}}
\put(0, -10){\line(1, 0){50}}
\put(0, 10){\line(0, -1){20}}
\put(50, 10){\line(0, -1){20}}
\put(15, -3){$r-1$}

\put(10, 10){\line(0, 1){15}}
\put(18, 15){$\cdots$}
\put(40, 10){\line(0, 1){15}}

\put(10, -10){\line(0, -1){15}}
\put(18, -20){$\cdots$}
\put(40, -10){\line(0, -1){15}}

\put(60, -25){\line(0, 1){50}}

\put(70, -3){$-$}

\end{picture}
\begin{picture}(80, 60)(-70, -30)

\put(-60, -3){$ (r-2)!^{-1}$}

\put(0, 10){\line(1, 0){40}}
\put(0, 25){\line(1, 0){40}}
\put(0, 10){\line(0, 1){15}}
\put(40, 10){\line(0, 1){15}}
\put(12, 15){\tiny$r-1$}

\put(5, 10){\line(0, -1){20}}
\put(25, 10){\line(0, -1){20}}
\put(8, -3){$\cdots$}

\qbezier(35,10)(45, 5)(50, -35)
\qbezier(35,-10)(45, -5)(50, 35)

\put(0, -10){\line(1, 0){40}}
\put(0, -25){\line(1, 0){40}}
\put(0, -10){\line(0, -1){15}}
\put(40, -10){\line(0, -1){15}}
\put(12, -20){\tiny$r-1$}

\put(5, 25){\line(0, 1){10}}
\put(35, 25){\line(0, 1){10}}
\put(14, 28){$\cdots$}

\put(5, -25){\line(0, -1){10}}
\put(35, -25){\line(0, -1){10}}
\put(14, -33){$\cdots$}

\put(55, -35){.}
\end{picture}
\]

\item
\[
\begin{picture}(80, 60)(0, -30)
\put(0, 10){\line(1, 0){60}}
\put(0, -10){\line(1, 0){60}}
\put(0, 10){\line(0, -1){20}}
\put(60, 10){\line(0, -1){20}}
\put(25, -3){$r$}

\put(10, 10){\line(0, 1){15}}
\put(20, 15){$\cdots$}
\put(40, 10){\line(0, 1){15}}

\qbezier(50, -10)(60, -40)(65, 0)
\qbezier(50, 10)(60, 40)(65, 0)

\put(10, -10){\line(0, -1){15}}
\put(20, -20){$\cdots$}
\put(40, -10){\line(0, -1){15}}

\put(75, -3){$=$}
\end{picture}
\begin{picture}(80, 60)(-90, -30)
\put(-80, -5){$-(r-1-\delta)$}
\put(0, 10){\line(1, 0){60}}
\put(0, -10){\line(1, 0){60}}
\put(0, 10){\line(0, -1){20}}
\put(60, 10){\line(0, -1){20}}
\put(20, -3){$r-1$}

\put(10, 10){\line(0, 1){15}}
\put(23, 15){$\cdots$}
\put(50, 10){\line(0, 1){15}}

\put(10, -10){\line(0, -1){15}}
\put(23, -20){$\cdots$}
\put(50, -10){\line(0, -1){15}}

\put(60, -25){.}
\end{picture}
\]

\end{enumerate}

\end{lemma}

The next result shows that there are constraints upon the parameter $\delta$ arising from consistency questions in the category $\wt\CB(\delta)$.

\begin{lemma}\label{lem:dreln}
With the assumptions of Remark \ref{rem:nonzero}, it follows that $\delta$ satisfies the polynomial equation 
\be\label{eq:deltarel}
\delta(\delta-1)\dots(\delta-(m-1))=m!.
\ee

Furthermore, $\Delta_m^*\Delta_m=m!\in R$.
\end{lemma}

\begin{proof}
We shall compute $\Delta_m^*\Delta_m$ in two different ways. First,
observe that by inspection of the relevant diagrams, it is evident that
 $(\Delta_m\Delta_m^*)^2=(\Delta_m^*\Delta_m) (\Delta_m\Delta_m^*)$, where $\Delta_m^*\Delta_m$ is a scalar.
Thus, applying Relation (4) twice,
we see that $(\Delta_m^*\Delta_m)\Sigma_m=\Sigma_m^2=m!\Sigma_m\neq 0$.
Comparing coefficients of the non-zero element $\Sigma_m$, it follows that
\be\label{eq:dstd}
\Delta_m^*\Delta_m=m!.
\ee

Next, note that
we have the relation (i) depicted in Fig. 2.

\vspace{-2.8cm}
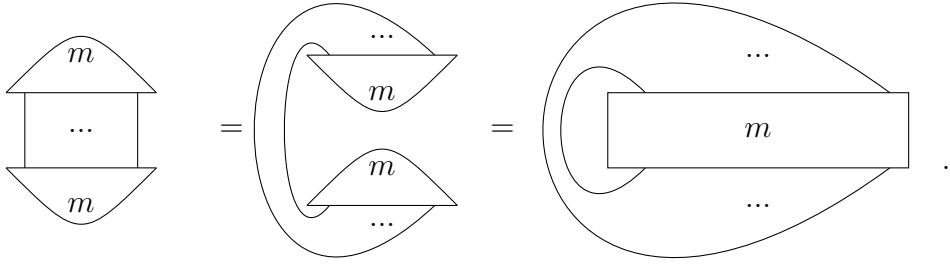
\begin{figure}[h]
\begin{tikzpicture}
\draw (-5,-.5).. controls (-4,0.5) .. (-3,-.5);
\draw (-5,-.5)-- (-3,-.5);

\draw (-5,-1.5).. controls (-4,-2.5) .. (-3,-1.5);
\draw (-5,-1.5)-- (-3,-1.5);

\node at (-4,0) {$m$};
\node at (-4,-2) {$m$};
\node at (-4,-1) {$...$};
\node at (-2,-1) {$=$};

\draw (-4.75,-.5)-- (-4.75,-1.5);
\draw (-3.25,-.5)-- (-3.25,-1.5);

\draw (-1,0).. controls (0,-1) .. (1,0);
\draw (-1,0)-- (1,0);
\draw (-.7,0).. controls (-1.5,1) and (-1.5,-3).. (-.7,-2);
\node at (0,-.5) {$m$};

\draw (-1,-2).. controls (0,-1) .. (1,-2);
\draw (-1,-2)-- (1,-2);
\draw (.7,0).. controls (-2.5,3) and (-2.5,-5).. (.7,-2);
\node at (0,-1.5) {$m$};
\node at (0,.25) {$...$};
\node at (0,-2.25) {$...$};

\node at (1.6,-1) {$=$};

\node at (5,-1) {$m$};
\node at (5,0) {$...$};
\node at (5,-2) {$...$};

\draw (3,-1.5) rectangle (7,-.5);

\draw (3.5,-1.5).. controls (2,-3) and (2,1).. (3.5,-.5);
\draw (6.75,-1.5).. controls (0.6,-6) and (0.6,4).. (6.75,-.5);
\node at (7.5,-1.5) {$.$};
\end{tikzpicture}
\vspace{-3cm}
\caption{Relation (i).}
\end{figure}

Now the right side of relation (i) is, by $m$ applications of Lemma \ref{lem:red}, equal to $(\delta-(m-1))(\delta-(m-2))\dots(\delta-1)\delta$,
while the left side is just $\Delta_m^*\Delta_m$. The result is now clear from \eqref{eq:dstd}.
\end{proof}

\begin{theorem}\label{thm:comp} The assumptions of Remark \ref{rem:nonzero} remain in force.
\begin{enumerate}
\item In the category $\wt\CB(\delta)$ we have the equality of morphisms
$\Sigma_{m+1}=f_m(\delta) \Sigma_m\ot I$. Here $m$ is the positive integer occurring in the definition of $\wt\CB(\delta)$
and $f_m$ is the polynomial in $\delta$ given by $f_m(\delta)=(\delta-(m-1))(\delta-(m-2))\dots(\delta-1)-(m-1)!$.
\item We have $\Sigma_{m+1}=0$.
\item We have $\delta=m$.
\end{enumerate}
\end{theorem}
\begin{proof} 
We begin by proving (1).
In this proof we shall make liberal use, both explicit and implicit, 
of the mutually inverse isomorphisms $\bU_p^q:\wt\CB_{p+q}^r\to \wt\CB_p^{r+q}$
and $\bA_q^r:\wt\CB_p^{r+q}\to\wt\CB_{p+q}^r$ defined in \cite[Cor. 2.8]{LZ5}.
Note that these isomorphisms involve only operations (tensor product and composition) 
with the Brauer morphisms in $\wt\CB(\delta)$.

The relation (1) of Lemma \ref{lem:red} yields in our situation
\[
\begin{picture}(80, 60)(0, -30)
\put(0, 10){\line(1, 0){60}}
\put(0, -10){\line(1, 0){60}}
\put(0, 10){\line(0, -1){20}}
\put(60, 10){\line(0, -1){20}}
\put(20, -3){$m+1$}

\put(10, 10){\line(0, 1){15}}
\put(23, 15){$\cdots$}
\put(50, 10){\line(0, 1){15}}

\put(10, -10){\line(0, -1){15}}
\put(23, -20){$\cdots$}
\put(50, -10){\line(0, -1){15}}

\put(70, -2){$=$}
\end{picture}
\begin{picture}(80, 60)(-10, -30)
\put(0, 10){\line(1, 0){50}}
\put(0, -10){\line(1, 0){50}}
\put(0, 10){\line(0, -1){20}}
\put(50, 10){\line(0, -1){20}}
\put(15, -3){$m$}

\put(10, 10){\line(0, 1){15}}
\put(18, 15){$\cdots$}
\put(40, 10){\line(0, 1){15}}

\put(10, -10){\line(0, -1){15}}
\put(18, -20){$\cdots$}
\put(40, -10){\line(0, -1){15}}

\put(60, -25){\line(0, 1){50}}

\put(70, -3){$-$}

\end{picture}
\begin{picture}(80, 60)(-70, -30)

\put(-60, -3){$ (m-1)!^{-1}$}

\put(0, 10){\line(1, 0){40}}
\put(0, 25){\line(1, 0){40}}
\put(0, 10){\line(0, 1){15}}
\put(40, 10){\line(0, 1){15}}
\put(12, 15){\tiny$m$}

\put(5, 10){\line(0, -1){20}}
\put(25, 10){\line(0, -1){20}}
\put(8, -3){$\cdots$}

\qbezier(35,10)(45, 5)(50, -35)
\qbezier(35,-10)(45, -5)(50, 35)

\put(0, -10){\line(1, 0){40}}
\put(0, -25){\line(1, 0){40}}
\put(0, -10){\line(0, -1){15}}
\put(40, -10){\line(0, -1){15}}
\put(12, -20){\tiny$m$}

\put(5, 25){\line(0, 1){10}}
\put(35, 25){\line(0, 1){10}}
\put(14, 28){$\cdots$}

\put(5, -25){\line(0, -1){10}}
\put(35, -25){\line(0, -1){10}}
\put(14, -33){$\cdots$}

\put(55, -35){.}
\end{picture}
\]

We now replace each of the two rectangles in the second summand on the right side above by the left side of Figure 1.
A little manipulation then shows that the result will follow if we prove the relation (a) in Fig. 3.

\bigskip
\begin{figure}[h]
\begin{tikzpicture}
\draw (1,1).. controls (3,3) .. (5,1);
\draw (1,1)--(5,1);

\draw (1,-1).. controls (3,-3) .. (5,-1);
\draw (1,-1)--(5,-1);

\draw (1.25,-1)--(1.25,1);
\draw (4.3,-1)--(4.3,1);

\draw (4.75,1).. controls (5.5,0) .. (6.5,-2.5);
\draw (4.75,-1).. controls (5.5,0) .. (6.5,2.5);

\node at (3,1.5) {$m$};
\node at (3,-1.5) {$m$};
\node at (2.75,0) {$m-1$};

\node at (8,0) {$=$};
\node at (11,0) {$(\delta-(m-1))\dots(\delta-1)$};

\draw (14,-2.5)--(14,2.5);

\end{tikzpicture}
\caption{Relation (a)}
\end{figure}
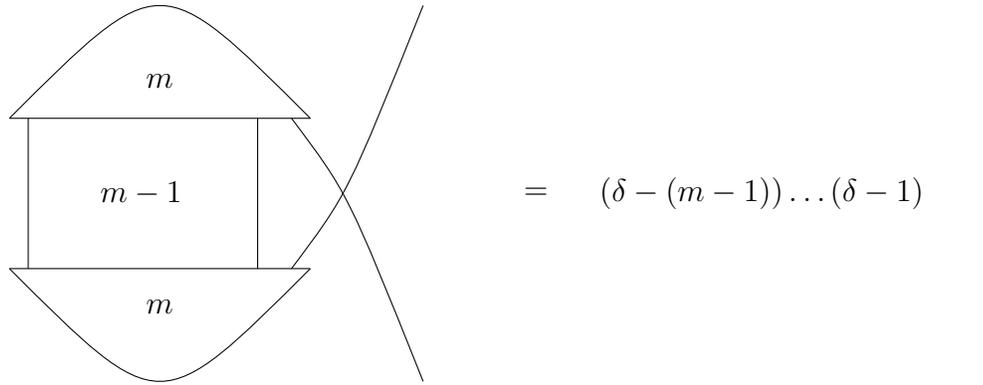

Next, observe that by rotating the top half of the left side of Relation (a) in Fig. 3 anticlockwise by 
$\frac{\pi}{2}$ and applying the isomorphism $\bU_0^1$
from $\CB_1^1$ to $\CB_0^2$, the relation (a) is equivalent to Relation (b) in Fig. 4.

\bigskip
\begin{figure}[h]
\begin{tikzpicture}
\draw (0,0).. controls (1,-2) .. (2,0);
\draw (0,0)--(2,0);

\draw (3,0).. controls (4,-2) .. (5,0);
\draw (3,0)--(5,0);

\draw (0.25,0)--(0.25,2);
\draw (4.75,0)--(4.75,2);

\draw (1.75,0).. controls (2.5,1.5) .. (3.25,0);
\draw (0.4,0).. controls (2.5,2.5) .. (4.6,0);

\node at (1,-1) {$m$};
\node at (4,-1) {$m$};
\node at (2.5,1.3) {$...m-1...$};

\node at (6,0) {$=$};
\node at (8.5,0) {$(\delta-(m-1))\dots(\delta-1)$};

\draw (11,2).. controls (11,-2) and (12.5,-2).. (12.5,2);

\end{tikzpicture}
\caption{Relation (b)}
\end{figure}
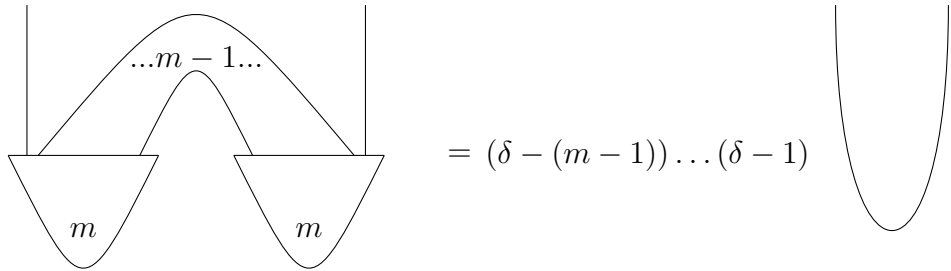

Now to prove Relation (b), observe first that applying the isomorphism $\bU_0^m$ to both sides of
the relation 4 as shown in Fig. 1, we obtain the relation (5) in Fig. 4.

\bigskip
\begin{figure}[h]
\begin{tikzpicture}
\draw (0,0).. controls (1,-2) .. (2,0);
\draw (0,0)--(2,0);

\draw (3,0).. controls (4,-2) .. (5,0);
\draw (3,0)--(5,0);

\draw (0.25,0)--(0.25,2);
\draw (4.75,0)--(4.75,2);

\draw (1.75,0) -- (1.75,2);
\draw (3.25,0) -- (3.25,2);

\node at (1,-1) {$m$};
\node at (4,-1) {$m$};

\node at (6,0) {$=$};

\draw (7,-0.5) rectangle (9,0.5);
\draw (7.25,.5) -- (7.25,2);
\draw (8.75,.5) -- (8.75,2);

\draw (8.75,-.5).. controls (9,-1.5) and (9.75,-1.5).. (10,2);
\draw (7.25,-.5).. controls (8,-2) and (11,-2).. (12,2);

\node at (1,1) {$...$};
\node at (4,1) {$...$};
\node at (8,0) {$m$};
\node at (10.7,1) {$...$};
\node at (8,1) {$...$};

\end{tikzpicture}
\caption{Relation (5)}
\end{figure}
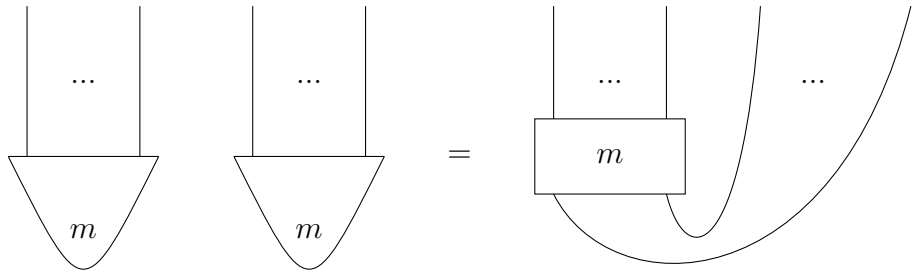

Finally, applying $I\ot A^{\ot m-1}\ot I$ to both sides of Relation (5), and applying Lemma \ref{lem:red} (2) $m-1$ times, 
we obtain the relation (b) of Fig. 3, and the proof of (1) is complete.

To prove (2), note that it follows from (1) that 
\[
\Sigma_{m+1}\circ (I^{\ot m-1}\ot U)=f_m(\delta) (\Sigma_m\ot I)\circ(I^{\ot m-1}\ot U).
\]
But the left side of this equation is evidently zero, while the right side is an invetible multiple of $f_m(\delta)\Sigma_m\ot I$.
It follows that $f_m(\delta)=0$, and hence by (1), that $\Sigma_{m+1}=0$.

Finally to prove (3), observe that $f_m(\delta)=0$ implies that $\delta(\delta-1)\dots(\delta-(m-1))=\delta(m-1)!$.
Comparing this to the relation $\delta(\delta-1)\dots(\delta-(m-1))=m!$ of Lemma \ref{lem:dreln}, we see that $\delta=m$.
\end{proof}

\begin{remark}\label{rem:zero} Although Theorem \ref{thm:comp} has been proved under the assumption that $\Sigma_m\neq 0$,
the conclusions (1) and (2) evidently remain true if $\Sigma_m=0$, but (3) fails. The application to invariant theory is also 
predicated upon this assumption.
\end{remark}

\section{A covariant functor-application to invariant theory.}

\subsection{The main theorem}
The results of the previous section indicate that the category of relevance to the invariant theory of $\SO_m$ is
$\wt\CB(m)$. Our main theorem is as follows.

\begin{theorem}\label{thm:main}
Let $V$ be $\C^m$, equipped with a nondegenerate symmetric bilinear form $(\;,\;)$, and let $G=\SO(V)$. Let $\CT$ be the full
tensor subcategory of finite dimensional representations of $G$ generated by $V$ under tensor product.

There is an {\em equivalence} of categories $\CF:\wt\CB(m)\lr\CT$ defined by $\CF(r)=V^{\ot r}$ for $r\in\N$,
$\CF(I)=\id:V\to V$, $\CF(X):V\ot V\to V\ot V$ is interchange of tensor factors, $\CF(A):V\ot V\to V$ is the map
$v\ot w\mapsto (v,w)$, $\CF(U):1\mapsto \check C$ and $\CF(\Delta_m)=\Lambda$, where $\Lambda$ is the harmonic homomorphism
in Example \ref{ex:1}.
\end{theorem}

Before giving the proof, we shall make some elementary observations concerning the structure of $\wt\CB(m)$.
\begin{lemma}\label{lem:str} Let $\wt\CB_0$ be the subcategory of $\wt\CB(m)$ generated by all Brauer 
diagrams (i.e. by the morphisms $I,X,A$ and $U$). 

\begin{enumerate}
\item Each diagram of $\wt\CB(m)$ is either in $\wt\CB_0$ or is obtained 
from $\Delta_m$ by tensoring and composing with elements of $\wt\CB_0$.
\item Let $s,t\in\N$. Then 
\[
\wt\CB_s^t=\wt\CB_{s,0}^t\oplus \wt\CB_{s,1}^t,
\]
where $\wt\CB_{s,0}^t$ is the span of the Brauer diagrams in $\wt\CB_s^t$, and 
$\wt\CB_{s,1}^t$ is the span of diagrams of the type described in (1).
\end{enumerate}
\end{lemma}
\begin{proof}[Proof of Lemma \ref{lem:str}]
If the diagram $D\in\wt\CB(m)$ is not in $\wt\CB_0$, then it may be expressed as a `word' in the generators $I,X,A,U$ and $\Delta_m$, 
with connectives $\ot$ (tensor product) and $\circ$ (composition), since $\Delta_m^*=\bA_m^0(\Delta_m)$. 
But the relation (5) in Fig. 5 above shows that any diagram with two consecutive
occurrences of $\Delta_m$, is equal in $\wt\CB$ to and element of $\wt\CB_0$. Thus we may assume that there is precisely one
occurrence of $\Delta_m$ in the word expression for $D$. This proves (1).

The statement (2) is an immediate consequence of (1), since each $\Hom$ space is spanned by diagrams, and the two
types of diagrams in (1) are complementary.
\end{proof}

\begin{proof}[Proof of Theorem \ref{thm:main}]
 We have seen \cite{LZ5} that the relations satisfied by $I,X,A$ and $U$ are satisfied by their images under the functor
$\CF$. The relations (1) to (4) of definition \ref{def:wbrcat} are clearly satisfied by $\Lambda$, so that we do have a functor.

It remains only to see that $\CF$ defines isomorphisms on $\Hom$ spaces. But Theorem \ref{thm:so-inv} states precisely 
that $\CF$ is surjective on $\Hom$ spaces (the FFT). We are therefore reduced to proving the injectivity of $\CF$ on $\Hom$
spaces, which is the SFT for $\SO_m$. 

By Lemma \ref{lem:str}(2), each element $\beta\in\ker(\CF_s^t:\Hom_{\wt\CB(m)}(s,t)\lr\Hom_G(V^{\ot s},V^{\ot t})$
is uniquely of the form $\beta=\beta_0+\beta_1$, where $\beta_i\in\wt\CB_{s,i}^t$ ($i=0,1$). Moreover $\CF_s^t$ maps
$\wt\CB_{s,0}^t$ to $\Hom_{\Or(V)}(V^{\ot s},V^{\ot t})$, and $\wt\CB_{s,1}^t$ to the space of skew invariants 
for $\Or(V)$. It follows that $\beta\in\ker(\CF_s^t)$ if and only if $\beta_i\in\ker(\CF_s^t)$ for $i=0,1$.

Now \cite[Thm. 4.8]{LZ5} states that the image of an element $\beta$ of $\wt\CB(m)_0$ under $\CF$ is zero if and only if 
$\beta$ is in the ideal $\langle\Sigma_{m+1}\rangle$
of morphisms generated under the operations of a tensor category by $\Sigma_{m+1}$.

By Theorem \ref{thm:comp}, this ideal of morphisms in zero in $\wt\CB(m)$. This proves that $\beta_0$ is in 
$\langle\Sigma_{m+1}\rangle$. As for $\beta_1$, note that because of its form, we have 
$\beta_1\circ(\Delta_m^*\ot I^{\ot r})\in\wt\CB_0$ and hence lies in $\langle\Sigma_{m+1}\rangle$
 for some $r$, and so $\beta_1=(m!)\inv \left(\beta_1\circ(\Delta_m^*\ot I^{\ot r})\right)\circ(\Delta_m\ot I^{\ot r})$
 is also in the ideal $\langle\Sigma_{m+1}\rangle$. 
 Hence $\beta\in\langle\Sigma_{m+1}\rangle$, and the proof is complete.
\end{proof}
\subsection{Dimensions of $\Hom$ spaces} The dimension of the space $\Hom_G(V^{\ot s},V^{\ot t})$, where $G$ is either 
$\Or(V)$ or $\SO(V)$, depends only on $r:=s+t$. Let $d(r)=\dim(\Hom_{\Or(V)}(\C,V^{\ot r})$
$(=\dim(\Hom_{\Or(V)}(V^{\ot s},V^{\ot t})$ when $s+t=r$). The following statement is easily deduced from Lemma 
\ref{lem:str} and Theorems \ref{thm:main} and \ref{thm:so-inv}.

\begin{corollary}\label{cor:dims}
If $s+t=r$, we have 
$$
\dim(\Hom_{\SO(V)}(V^{\ot s},V^{\ot t}))=\binom{r}{m} d(r-m)+d(r).
$$
\end{corollary}
Note that since $d(r)=0$ when $r$ is odd, depending on the parity of $m$, either or both terms on the right side 
of the above formula may vanish.


\end{document}